\documentclass[12pt]{article}

\usepackage[T1]{fontenc}
\usepackage{lmodern}
\usepackage{amsmath,amssymb,amsfonts,amsthm,mathtools}
\usepackage{enumitem}
\usepackage{needspace}
\usepackage{tikz}
\usepackage{float}

\usepackage[pdfauthor={Chunqiu Fang and Rongxing Xu},
  pdftitle={On the minimum number of triangles in balanced tripartite graphs with large minimum degree},
  pdfsubject={Extremal graph theory},
  pdfstartview=XYZ,bookmarks=true,colorlinks=true,
  linkcolor=blue,urlcolor=blue,citecolor=blue,
  linktocpage=true,hyperindex=true]{hyperref}

\allowdisplaybreaks
\numberwithin{equation}{section}
\setlist[enumerate]{label=(\arabic*),leftmargin=2.35em,itemsep=0.2em,topsep=0.3em}
\setlist[itemize]{leftmargin=2em,itemsep=0.2em,topsep=0.3em}

\newtheorem{theorem}{Theorem}[section]
\newtheorem{conjecture}[theorem]{Conjecture}

\newtheorem{lemma}[theorem]{Lemma}
\newtheorem{claim}{Claim}[section]

\title{\large{\bfseries On the minimum number of triangles in balanced tripartite graphs with large minimum degree}}
\author{
  Chunqiu Fang \thanks{School of Computer Science and Technology, Dongguan University of Technology, Dongguan, Guangdong, 523808, China. Email: \texttt{fcq15@tsinghua.org.cn}.}
  \quad
  Rongxing Xu\thanks{School of Mathematical Sciences, Zhejiang Normal University, Jinhua, Zhejiang, 321000, China. Email: \texttt{xurongxing@zjnu.edu.cn}.}
}
\date{}

\begin{document}
\maketitle

\begin{abstract}
Let $f(n,t)$ be the minimum number of triangles in a tripartite graph with $n$ vertices in each part and minimum degree at least $n+t$. In 1975, Bollob\'{a}s, Erd\H{o}s and Szemer\'{e}di proved that $f(n,1)=\min\{4,n\}$. They further remarked that it is ``very likely'' that $f(n,t)\ge4t^3$ for $n\ge5t$. They also proved that $f(n,t) \ge t^3$ for all integers $n \ge t \ge 1$. We construct graphs showing that, for all integers $t\ge1$ and $n\ge3t+2\lceil(1+\sqrt5)t/2\rceil$,
\[
f(n,t)\le(1+\sqrt5)t^3+\left(1+\frac1{\sqrt5}\right)t^2.
\]
Here $1+\sqrt5\approx3.236<4$, and the displayed upper bound is strictly less than $4t^3$ for every $t\ge2$, disproving their proposed bound.
We also improve their lower bound $t^3$ by showing that $f(n,t)\ge\frac{12}{5}t^3$ for all integers $t\ge2$ and $n\ge18t^6$.
\end{abstract}

\section{Introduction}

Let $\mathcal{G}_3(n)$ denote the family of tripartite graphs with $n$ vertices in each part. Triangle counts have been studied under various restrictions on graphs in this family. For example, Fischer and Matou\v{s}ek~\cite{FM01} posed the problem of maximizing the number of triangles when the bipartite graph induced by every pair of parts is $C_4$-free; see~\cite{CMT18,FX26} for subsequent progress. Bondy et al.~\cite{BSTT06} studied edge-density conditions between pairs of parts that force a triangle, and Baber, Johnson and Talbot~\cite{BJT10} subsequently studied the minimum triangle density under prescribed edge densities.

Bollob\'{a}s, Erd\H{o}s and Szemer\'{e}di~\cite{BES75} studied complete subgraphs in multipartite graphs with parts of equal size under minimum-degree conditions. In the tripartite case, every $G\in\mathcal{G}_3(n)$ with $\delta(G)>n$ contains a triangle. This threshold is sharp: joining one part completely to the other two and leaving no edges between those two parts gives a triangle-free graph of minimum degree $n$. Let $\tau(G)$ denote the number of triangles in $G$. For integers $1\le t\le n$, their counting problem is to determine
\[
f(n,t)=\min\{\tau(G):G\in\mathcal{G}_3(n),\ \delta(G)\ge n+t\}.
\]

Bollob\'{a}s, Erd\H{o}s and Szemer\'{e}di~\cite{BES75} proved $f(n,1)=\min\{4,n\}$. For $n\ge5t$, they constructed graphs with exactly $4t^3$ triangles. They suggested the following conjecture, which also appears as a question in the UCSD Erd\H{o}s problems collection.\footnote{In~\cite[p.~101]{BES75}, the authors wrote: ``It is very likely that every graph $G_3(n)$, $n\ge5t$, with minimal degree $n+t$ contains at least $4t^3$ triangles''. See also item~9 in the ``Random graphs and graph enumeration'' section of the \href{https://mathweb.ucsd.edu/~erdosproblems/erdos/}{collection} and the \href{https://mathweb.ucsd.edu/~erdosproblems/erdos/newproblems/TrianglesInMultipartiteGraph.html}{individual problem page}.}

\begin{conjecture}[Bollob\'{a}s, Erd\H{o}s and Szemer\'{e}di~\cite{BES75}]\label{conj:BES}
For all positive integers $n$ and $t$ with $n\ge5t$, $f(n,t)\ge4t^3$.
\end{conjecture}

Bollob\'{a}s, Erd\H{o}s and Szemer\'{e}di~\cite{BES75} proved the following lower bound.

\begin{theorem}[\cite{BES75}]\label{thm:BES}
For all integers $1\le t\le n$, $f(n,t)\ge t^3$.
\end{theorem}

Further progress on lower bounds was made by Chen et al.~\cite{CHLLMZ25}, who proved $f(n,t)\ge n^2(3t-n)/2$, with equality when $n$ is even and $t\ge n/2$.

In this paper, we disprove Conjecture~\ref{conj:BES} for every fixed $t\ge2$ and sufficiently large $n$. We also strengthen the lower bound in Theorem~\ref{thm:BES} in this range. Our first result is the following upper bound, obtained by an explicit construction.

\begin{theorem}\label{thm:upper}
Let $\varphi=(1+\sqrt5)/2$. For all integers $t\ge1$ and $n\ge3t+2\lceil\varphi t\rceil$,
\[
f(n,t)\le
2\min_{p\in\{\lfloor\varphi t\rfloor,\lceil\varphi t\rceil\}}
p\max\{t^2,(p+t)(2t-p)\}
\le(1+\sqrt5)t^3+\left(1+\frac1{\sqrt5}\right)t^2.
\]
\end{theorem}

Here $1+\sqrt5\approx3.236<4$. For every integer $t\ge2$, the upper bound in Theorem~\ref{thm:upper} is strictly less than $4t^3$, so Conjecture~\ref{conj:BES} fails.
Our second result improves the lower bound $t^3$ by a constant factor when $n$ is sufficiently large in terms of $t$.

\begin{theorem}\label{thm:lower}
For all integers $t\ge2$ and $n\ge18t^6$, $f(n,t)\ge\frac{12}{5}t^3$.
\end{theorem}

The explicit threshold $n \ge 18t^6$ in the lower bound is not optimized. More involved estimates can slightly improve the lower bound, but we retain $12/5$ here to keep the proof short and accessible.

In Section~\ref{sec:construction}, we give the construction proving Theorem~\ref{thm:upper}. To prove Theorem~\ref{thm:lower}, we partition the graph into two vertex-disjoint induced subgraphs in Section~\ref{sec:decomposition} and apply a local triangle estimate in Section~\ref{sec:local}.

We conclude this section with some notation and conventions. All graphs are finite and simple. We write $d_U(v)$ (or $d_H(v,U)$ when the graph needs to be specified) for the number of neighbors of $v$ in $U$, $e_H(U,W)$ for the number of edges between disjoint subsets $U,W$, and $\tau_H(e)$ for the number of triangles containing an edge $e$. Indices of three parts are always read modulo three.

\section{Proof of Theorem~\ref{thm:upper}}\label{sec:construction}

Let $t,p,n$ be positive integers such that $t\le p\le2t$ and $n\ge2p+3t$.

We first construct a bipartite graph $H=(X,Y)$ with parts $X=\{x_0,\ldots,x_{p+t-1}\}$ and $Y=\{y_0,\ldots,y_{t-1}\}$ by adding edges in two steps:
\begin{enumerate}
\item For each $0\le j<t$, join $y_j$ to $x_{jt},x_{jt+1},\ldots,x_{jt+t-1}$, with indices taken modulo $p+t$. Since $t\le p+t$, these neighbors are distinct. This step creates exactly $t^2$ edges and gives every vertex of $Y$ degree $t$. Before reduction modulo $p+t$, the subscripts run through $0,1,\ldots,t^2-1$, so every vertex of $X$ has degree $\lfloor t^2/(p+t)\rfloor$ or $\lceil t^2/(p+t)\rceil$.
\item For each vertex of $X$ whose degree is less than $2t-p$, add edges to its nonneighbors in $Y$ until its degree reaches $2t-p$. This is possible because $2t-p\le t=|Y|$.
\end{enumerate}

We now count the edges of $H$.
\begin{itemize}
\item If $t^2\ge(p+t)(2t-p)$, every vertex of $X$ already has degree at least $2t-p$ after the first step. Hence the second step adds no edges, and $|E(H)|=t^2$.

\item If $t^2<(p+t)(2t-p)$, every vertex of $X$ has degree at most $2t-p$ after the first step. Thus after the second step, every vertex of $X$ has degree exactly $2t-p$, and $|E(H)|=(p+t)(2t-p)$.
\end{itemize}

Consequently, $|E(H)|=\max\{t^2,(p+t)(2t-p)\}$. By construction, every vertex of $X$ has degree at least $2t-p$, and every vertex of $Y$ has degree at least $t$.

Let $A_1,A_2,A_3,B_1,B_2,B_3,C_1,C_2,C_3$ be pairwise disjoint sets with $|A_1|=|A_3|=p$, and
\[
|A_2|= n-2p, \quad 
|B_1|=|C_1|=p+t,\quad
|B_2|=|C_2|=t,\quad
|B_3|=|C_3|=n-p-2t.
\]
These sizes are positive by the assumption $n\ge2p+3t$. 

We construct $G\in\mathcal{G}_3(n)$ with parts $A=A_1\cup A_2\cup A_3$, $B=B_1\cup B_2\cup B_3$, and $C=C_1\cup C_2\cup C_3$. The edges are given as follows (see Fig.~\ref{fig:construction} for an illustration).\footnote{When $p=t$, we have $H=K_{t,2t}$ and recover a special case of the construction of Bollob\'{a}s, Erd\H{o}s and Szemer\'{e}di~\cite{BES75} (pp.~100--101). For $n \ge 5t$, the resulting graph has exactly $4t^3$ triangles.}
\begin{enumerate}
	\item Add all edges between $A_1$ and $B\cup C_2$, between $A_2$ and $B_2\cup B_3\cup C_2\cup C_3$, and between $A_3$ and $C\cup B_2$.
	\item Add all edges between $B_1$ and $C_1\cup C_3$, and between $B_3$ and $C_1$.
	\item Place a copy of $H$ between $B_1$ and $C_2$, with $B_1$ corresponding to $X$ and $C_2$ to $Y$. Also place a copy of $H$ between $C_1$ and $B_2$, with $C_1$ corresponding to $X$ and $B_2$ to $Y$.
\end{enumerate}

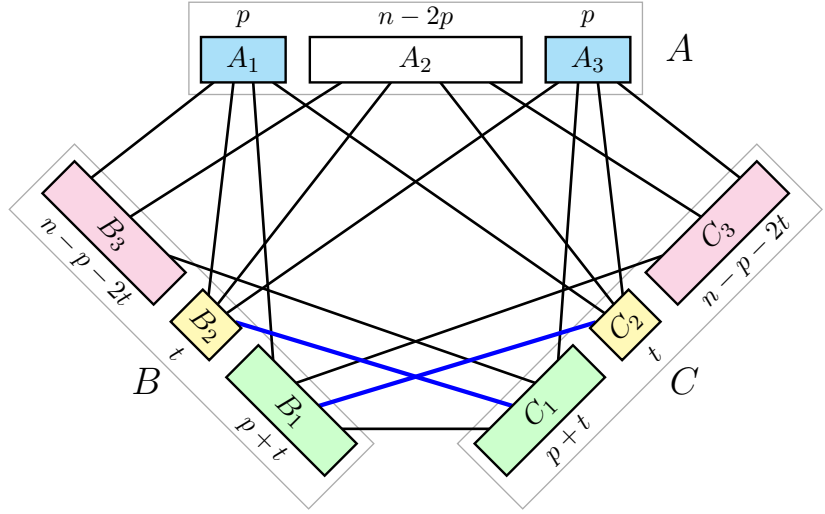
\begin{figure}[h]
\centering
\begin{tikzpicture}[
  x=1cm,y=1cm,
  every node/.style={font=\small},
  complete/.style={black,line width=1pt},
  copy H/.style={blue,line width=1.6pt},
  part frame/.style={draw=gray!65,line width=0.5pt},
  vertex set/.style={rectangle,draw=black,fill=white,line width=0.9pt,
    minimum height=0.60cm,inner sep=1.5pt}
]
\pgfmathsetmacro{\partangle}{45}
\coordinate (A1) at (-2.28,2.48);
\coordinate (A2) at (0,2.48);
\coordinate (A3) at (2.28,2.48);
\begin{scope}[shift={(-3,-0.8)},rotate=-\partangle]
  \coordinate (B1) at (1.92,0);
  \coordinate (B2) at (0.32,0);
  \coordinate (B3) at (-1.4,0);
  \draw[part frame] (-2.6,-0.76) rectangle (3,0.46);
\end{scope}
\begin{scope}[shift={(3,-0.8)},rotate=\partangle]
  \coordinate (C1) at (-1.92,0);
  \coordinate (C2) at (-0.32,0);
  \coordinate (C3) at (1.4,0);
  \draw[part frame] (-3,-0.76) rectangle (2.6,0.46);
\end{scope}
\draw[part frame] (-3,2.02) rectangle (3,3.24);

\foreach \v/\ang/\face/\targets in {
  A1/0/-0.3/{B3/-0.392,B2/-0.130666,B1/0.130666,C2/0.392},
  A2/0/-0.3/{B3/-0.98,B2/-0.326666,C2/0.326666,C3/0.98},
  A3/0/-0.3/{B2/-0.392,C1/-0.130666,C2/0.130666,C3/0.392},
  B1/-\partangle/0.3/{A1/-0.644,C3/-0.214666,C2/0.214666,C1/0.644},
  B2/-\partangle/0.3/{A1/-0.252,A2/-0.084,A3/0.084,C1/0.252},
  B3/-\partangle/0.3/{A1/-0.728,A2/0,C1/0.728},
  C1/\partangle/0.3/{B1/-0.644,B2/-0.214666,B3/0.214666,A3/0.644},
  C2/\partangle/0.3/{B1/-0.252,A1/-0.084,A2/0.084,A3/0.252},
  C3/\partangle/0.3/{B1/-0.728,A2/0,A3/0.728}}
{
  \begin{scope}[shift={(\v)},rotate=\ang]
    \foreach \target/\offset in \targets
      \coordinate (\v-\target) at (\offset,\face);
  \end{scope}
}

\foreach \u/\v in {
  A1/B1,A1/B2,A1/B3,A1/C2,
  A2/B2,A2/B3,A2/C2,A2/C3,
  A3/C1,A3/C2,A3/C3,A3/B2,
  B1/C1,B1/C3,B3/C1}
  \draw[complete] (\u-\v) -- (\v-\u);

\draw[copy H] (B1-C2) -- (C2-B1);
\draw[copy H] (C1-B2) -- (B2-C1);

\foreach \v/\lab/\ang/\len/\size/\side/\shade in {
  A1/A_1/0/1.12/p/0.544/cyan!30,
  A2/A_2/0/2.8/{n-2p}/0.544/white,
  A3/A_3/0/1.12/p/0.544/cyan!30,
  B1/B_1/-\partangle/1.84/{p+t}/-0.544/green!20,
  B2/B_2/-\partangle/0.72/t/-0.544/yellow!35,
  B3/B_3/-\partangle/2.08/{n-p-2t}/-0.544/magenta!20,
  C1/C_1/\partangle/1.84/{p+t}/-0.544/green!20,
  C2/C_2/\partangle/0.72/t/-0.544/yellow!35,
  C3/C_3/\partangle/2.08/{n-p-2t}/-0.544/magenta!20}
{
  \node[vertex set,fill=\shade,minimum width=\len cm,rotate=\ang] at (\v) {};
  \node[rotate=\ang,inner sep=0pt] at (\v) {$\lab$};
  \path (\v) ++({-\side*sin(\ang)},{\side*cos(\ang)})
    node[font=\footnotesize,rotate=\ang,inner sep=0pt] {$\size$};
}
\node[font=\large,anchor=west] at (3.16,2.63) {$A$};
\node[font=\large] at (-3.56,-1.768) {$B$};
\node[font=\large] at (3.56,-1.768) {$C$};
\end{tikzpicture}
\caption{The construction of $G$. Black lines indicate complete bipartite graphs, and blue lines indicate copies of $H$. Matching fills indicate equal sizes.}
\label{fig:construction}
\end{figure}

Since $n\ge2p+3t$, it is straightforward to check that every $v\in A\cup B_3\cup C_3$ satisfies $d_G(v)\ge n+t$, and we omit the verification.
By the constructions of $H$ and $G$, every vertex of $B_1$ has at least $2t-p$ neighbors in $C_2$, and every vertex of $C_1$ has at least $2t-p$ neighbors in $B_2$. Hence, for every $v\in B_1\cup C_1$,
\[
d_G(v)\ge p+(p+t)+(n-p-2t)+(2t-p)=n+t.
\]
Similarly, every vertex of $B_2$ has at least $t$ neighbors in $C_1$, and every vertex of $C_2$ has at least $t$ neighbors in $B_1$. Since these vertices are also adjacent to every vertex of $A$, every $v\in B_2\cup C_2$ satisfies $d_G(v)\ge |A|+t=n+t$. Thus $\delta(G)\ge n+t$.

It is easy to see that every triangle of $G$ lies in $G[A_1\cup B_1\cup C_2]$ or $G[A_3\cup B_2\cup C_1]$. Thus $\tau(G)=2p|E(H)|=2p\max\{t^2,(p+t)(2t-p)\}$.

To obtain the upper bounds in Theorem~\ref{thm:upper}, we now choose $p$ to make the triangle count in our construction as small as possible. The two expressions inside the maximum are equal at $p=\varphi t$, so we consider the two integer choices $p_1=\lfloor\varphi t\rfloor$ and $p_2=\lceil\varphi t\rceil$. Both values lie in $[t,2t]$. Since $n\ge3t+2p_2$ by hypothesis, the condition $n\ge2p+3t$ holds for either choice. Thus both choices give valid constructions, and taking the smaller of their triangle counts proves the first inequality.

To prove the second inequality, we estimate the error introduced by rounding $\varphi t$ to an integer. Let $\varepsilon=\varphi t-p_1$. Then $0<\varepsilon<1$, $p_1=\varphi t-\varepsilon$ and $p_2=\varphi t+1-\varepsilon$. Thus
\[
\begin{aligned}
p_1\max\{t^2,(p_1+t)(2t-p_1)\}
& = p_1(p_1+t)(2t-p_1) \\
& = \varphi t^3+\varphi^2\varepsilon t^2+(1-3\varphi)\varepsilon^2t+\varepsilon^3 \le\varphi t^3+\varphi^2\varepsilon t^2,
\end{aligned}
\]
because $\varepsilon^2(\varepsilon-(3\varphi-1)t)<0$.
Similarly,
\[
p_2\max\{t^2,(p_2+t)(2t-p_2)\}=p_2t^2=\varphi t^3+(1-\varepsilon)t^2.
\]
The rounding terms satisfy
\[
\min\{\varphi^2\varepsilon,1-\varepsilon\}
\le\frac{\varphi^2}{1+\varphi^2}
=\frac{5+\sqrt5}{10}.
\]
Taking the smaller of the two bounds above and multiplying by $2$, we obtain
\[
\begin{aligned}
2\min_{p\in\{p_1,p_2\}}p\max\{t^2,(p+t)(2t-p)\}
&\le2\varphi t^3+2\min\{\varphi^2\varepsilon,1-\varepsilon\}t^2\\
&\le2\varphi t^3+\frac{5+\sqrt5}{5}t^2 =(1+\sqrt5)t^3+\left(1+\frac1{\sqrt5}\right)t^2.
\end{aligned}
\]
This completes the proof of Theorem~\ref{thm:upper}.

\section{A partition lemma for the proof of Theorem~\ref{thm:lower}}\label{sec:decomposition}

In this section, we present a key partition lemma that will be used in the proof of Theorem~\ref{thm:lower}. Before stating it, we introduce the necessary concepts.

A directed graph $D$ is \emph{strongly connected} if every vertex can be reached from every other vertex by a directed path. A \emph{strong component} of $D$ is a maximal induced subdigraph of $D$ which is strongly connected.

We shall use the following fact about strong components.

\begin{lemma}[{\cite{BG00}, p.~17}]\label{lem:strong-order}
	The vertex sets of the strong components of any directed graph can be ordered as $D_1,\ldots,D_k$ for some $k$ so that every arc between distinct sets is directed from $D_i$ to $D_j$ with $i<j$.
\end{lemma}

We call such an ordering an \emph{acyclic ordering} of the strong components. For a fixed such ordering and an integer $s$ with $0\le s\le k$, the union $S=\bigcup_{r=1}^{s}D_r$ of the first $s$ component vertex sets is called an \emph{initial segment}. We take $S=\varnothing$ when $s=0$.

For a directed graph $F$, let $V(F)$ and $A(F)$ denote its vertex set and arc set, respectively. For $u\in V(F)$, let $N_F^+(u)$ and $N_F^-(u)$ denote the sets of its out-neighbors and in-neighbors, respectively. Its out-degree and in-degree are $d_F^+(u)=|N_F^+(u)|$ and $d_F^-(u)=|N_F^-(u)|$, respectively. An arc from $u$ to $v$ is written as $(u,v)$. We use $\vec P$ and $\vec C$ for directed paths and cycles, and $P$ and $C$ for their underlying undirected graphs, respectively.

Let $H$ be a tripartite graph, and let $(V_0,V_1,V_2)$ be a fixed ordering of its three parts. We write $H=(V_0,V_1,V_2)$ to record this order. Let $\vec{H}$ denote the directed graph obtained from $H$ by directing its edges from $V_0$ to $V_1$, from $V_1$ to $V_2$, and from $V_2$ to $V_0$. For each $i\in\{0,1,2\}$ and every $u\in V_i$, we have
\[
d_{\vec{H}}^+(u)=d_{V_{i+1}}(u),\quad d_{\vec{H}}^-(u)=d_{V_{i-1}}(u),\quad d_H(u)=d_{\vec{H}}^+(u)+d_{\vec{H}}^-(u).
\]
Every triangle of $H$ becomes a directed cycle in $\vec{H}$, so its three vertices lie in the same strong component.

For a real number $r>0$, a triple $(A_1,A_2,A_3)$ of finite sets is \emph{$r$-unbalanced} if $\max_{1\le i\le3}|A_i|-\min_{1\le i\le3}|A_i|\ge r$.

We now state the partition lemma.

\begin{lemma}\label{lem:prefix}
Let $t$ be a positive integer, and let $G=(V_0,V_1,V_2)\in\mathcal{G}_3(n)$ satisfy $\delta(G)\ge n+t$. If
$
n>3\tau(G)^2+2\tau(G)-2,
$
then each $V_i$ admits a partition $V_i=V_i^1\cup V_i^2$ such that
\begin{enumerate}[label=(C\arabic*),ref=(C\arabic*),leftmargin=3em]
\item\label{item:partition-nonempty} For all $i\in\{0,1,2\}$ and $j\in\{1,2\}$, $V_i^j\ne\varnothing$.
\item\label{item:partition-forward} For each $i\in\{0,1,2\}$ and every $u\in V_i^1$, $d_{G_1}(u)\ge |V_{i+1}^1|+t$, where $G_1=G[V_0^1\cup V_1^1\cup V_2^1]$.
\item\label{item:partition-backward} For each $i\in\{0,1,2\}$ and every $u\in V_i^2$, $d_{G_2}(u)\ge |V_{i-1}^2|+t$, where $G_2=G[V_0^2\cup V_1^2\cup V_2^2]$.
\item\label{item:partition-imbalance} For each $j\in\{1,2\}$, the triple $(V_0^j,V_1^j,V_2^j)$ is $t$-unbalanced.
\end{enumerate}
We call any partition satisfying \ref{item:partition-nonempty}--\ref{item:partition-imbalance} a \emph{valid partition} of $G$.
\end{lemma}

\begin{proof}
We take the parts of $G$ in the order $(V_0,V_1,V_2)$ and let $\vec{G}$ denote the corresponding orientation. By Lemma~\ref{lem:strong-order}, we may list the vertex sets of its strong components as $D_1,\ldots,D_k$ in an acyclic ordering.

\begin{samepage}
\begin{claim}\label{clm:initial-segment-partition}
If there exists an initial segment $S$ such that $(S\cap V_0,S\cap V_1,S\cap V_2)$ is $t$-unbalanced, then, for $i=0,1,2$, the sets $V_i^1=S\cap V_i$ and $V_i^2=V_i\setminus S$ form a valid partition of $G$.
\end{claim}
\end{samepage}

\begin{proof}
Since $|V_0|=|V_1|=|V_2|=n$ and $(S\cap V_0,S\cap V_1,S\cap V_2)$ is $t$-unbalanced, $S$ is nonempty and proper.

For $u\in V_i^1$, we have $N_{\vec{G}}^-(u)\subseteq S$, since no arc enters $S$ from its complement. Thus its neighbors outside $S$ form the set $N_{\vec{G}}^+(u)\setminus S\subseteq V_{i+1}^2$. Hence $u$ loses at most $|V_{i+1}^2|$ neighbors when we restrict to $G_1$, that is, 
\[
d_{G_1}(u)\ge n+t-|V_{i+1}^2|=|V_{i+1}^1|+t.
\]
For $u\in V_i^2$, we have $N_{\vec{G}}^+(u)\cap S=\varnothing$, since no arc enters $S$ from its complement. Thus its neighbors in $S$ form the set $N_{\vec{G}}^-(u)\cap S\subseteq V_{i-1}^1$. Restricting to $G_2$, we have 
\[
d_{G_2}(u)\ge n+t-|V_{i-1}^1|=|V_{i-1}^2|+t.
\]
Thus the partition satisfies \ref{item:partition-forward} and \ref{item:partition-backward}. 

Since $S$ is nonempty and proper, it follows that both $G_1$ and $G_2$ are nonempty. Hence \ref{item:partition-nonempty} holds. Indeed, if $V_i^1=\emptyset$, then \ref{item:partition-forward} would force $V_{i+1}^1=\emptyset$ and then $V_{i+2}^1=\emptyset$, contradicting the fact that $G_1$ is nonempty. Applying the same argument in reverse cyclic order to $G_2$ using \ref{item:partition-backward}, we see that all six sets $V_i^j$ are nonempty.

By hypothesis, Condition~\ref{item:partition-imbalance} holds for $j=1$. Since $|V_i^2|=n-|V_i^1|$, 
\[
\max_i|V_i^2|-\min_i|V_i^2|
=\max_i|V_i^1|-\min_i|V_i^1|\ge t,
\]
so \ref{item:partition-imbalance} also holds for $j=2$.
\end{proof}

By Claim~\ref{clm:initial-segment-partition}, it remains to find such an initial segment. Suppose, to the contrary, that none exists. Since $t$ and the part sizes are integers, every initial segment $R$ then satisfies $\max_i|R\cap V_i|-\min_i|R\cap V_i|\le t-1$.

\begin{claim}\label{clm:nontrivial-components}
For each $i\in\{1,2,\ldots,k\}$, $|D_i|\ge2$.
\end{claim}

\begin{proof}
Suppose that $D_r=\{u\}$ with $u\in V_i$. Let $R=\bigcup_{s=1}^{r-1}D_s$, the union of the first $r-1$ component vertex sets in the fixed acyclic ordering. The acyclic ordering implies $N_{\vec{G}}^-(u)\subseteq R\cap V_{i-1}$ and $N_{\vec{G}}^+(u)\subseteq V_{i+1}\setminus R$. Hence
\[
d_G(u)\le |R\cap V_{i-1}|+n-|R\cap V_{i+1}|\le n+t-1,
\]
contrary to $\delta(G)\ge n+t$.
\end{proof}

For an arc $(u,v)$ with $u\in V_i$, we have $N_{\vec{G}}^-(u)\cup N_{\vec{G}}^+(v)\subseteq V_{i-1}$. Each vertex in their intersection forms a triangle containing $uv$, so $\tau_G(uv)=|N_{\vec{G}}^-(u)\cap N_{\vec{G}}^+(v)|$. Since $|V_{i-1}|=n$ and $d_{\vec{G}}^-(u)\ge n+t-d_{\vec{G}}^+(u)$, we have
\begin{equation}\label{eq:global-edge-count}
\tau_G(uv)\ge d_{\vec{G}}^-(u)+d_{\vec{G}}^+(v)-n
\ge t+d_{\vec{G}}^+(v)-d_{\vec{G}}^+(u).
\end{equation}


For $U\subseteq V(G)$, we write $\tau(U)$ for $\tau(G[U])$ when $G$ is clear. In particular, $\tau(V(G))=\tau(G)$.

\begin{claim}\label{clm:degree-spread}
For each $i\in\{1,\ldots,k\}$, $\tau(D_i)\ge t$ and
\begin{equation}\label{eq:component-spread}
\max_{u\in D_i}d_{\vec{G}}^+(u)-\min_{v\in D_i}d_{\vec{G}}^+(v)\le2\tau(D_i)-t.
\end{equation}
\end{claim}

\begin{proof}
By Claim~\ref{clm:nontrivial-components}, $|D_i|\ge2$, so the strong component with vertex set $D_i$ contains a simple directed cycle $\vec C$ with underlying undirected graph $C$. Since $\vec{G}$ has no loops or pairs of opposite arcs, $|E(C)|\ge3$.
We double-count pairs $(e,T)$, where $e\in E(C)$ and $T$ is a triangle of $G[D_i]$ containing $e$. Every triangle of $G$ containing an edge of $C$ lies in $G[D_i]$, so for each $e\in E(C)$, $\tau_G(e)=\tau_{G[D_i]}(e)$. 
Counting by edges and applying Inequality~\eqref{eq:global-edge-count}, we have 
\[
\sum_{(u,v)\in A(\vec C)}\tau_G(uv)
\ge |E(C)|t+\sum_{(u,v)\in A(\vec C)}\bigl(d_{\vec{G}}^+(v)-d_{\vec{G}}^+(u)\bigr) =|E(C)|t.
\]
Counting by triangles, each triangle contributes at most three pairs. Thus the number of pairs is at most $3\tau(D_i)$. Thus we have $3\tau(D_i) \geq 3t$, and $\tau(D_i)\ge t$ follows.

If $d_{\vec{G}}^+$ is not constant on $D_i$, choose $x,y\in D_i$ minimizing and maximizing $d_{\vec{G}}^+$, respectively. The strong component with vertex set $D_i$ contains a simple directed path $\vec P$ from $x$ to $y$ with underlying undirected graph $P$. As $x\ne y$, we have $|E(P)|\ge1$. 
We double-count pairs $(e,T)$, where $e\in E(P)$ and $T$ is a triangle of $G[D_i]$ containing $e$. Every triangle of $G$ containing an edge of $P$ lies in $G[D_i]$. 
Counting by edges and using Inequality~\eqref{eq:global-edge-count}, we have
\[  
\sum_{(u,v)\in A(\vec P)}\tau_G(uv) \ge |E(P)|t+\sum_{(u,v)\in A(\vec P)}\bigl(d_{\vec{G}}^+(v)-d_{\vec{G}}^+(u)\bigr) \ge t+d_{\vec{G}}^+(y)-d_{\vec{G}}^+(x). 
\] 
Counting by triangles, each triangle contributes at most two pairs as $P$ is a path, the number of pair is at most $2\tau(D_i)$. Together with the preceding inequality, Inequality~\eqref{eq:component-spread} follows. If $d_{\vec{G}}^+$ is constant on $D_i$, Inequality~\eqref{eq:component-spread} follows from $\tau(D_i)\ge t$.
\end{proof}

\begin{claim}\label{clm:few-nonneighbors}
For $r\in\{1,\ldots,k\}$ and $i\in\{0,1,2\}$, every vertex of $D_r\cap V_i$ has at most $3\tau(D_r)-2$ nonneighbors in $D_r\cap V_{i+1}$.
\end{claim}

\begin{proof}
Each triangle in $G[D_r]$ contains one vertex from each part, so at most $\tau(D_r)$ vertices in each part belong to triangles. If every vertex in $D_r$ belongs to a triangle, then $|D_r\cap V_{i+1}|\le \tau(D_r)\le3\tau(D_r)-2$, since Claim~\ref{clm:degree-spread} implies $\tau(D_r)\ge t\ge1$. The claim holds in this case. 

We may therefore assume that some vertex in $D_r$ belongs to no triangle. Choose such a vertex $x$ with maximum $d_{\vec{G}}^+(x)$, and suppose that $x\in V_j$.

We have $|N_{\vec{G}}^-(x)\cap D_r|\le \tau(D_r)$. Indeed, $N_{\vec{G}}^-(x)\cap D_r\subseteq V_{j-1}$, so it suffices to show that each of these vertices belongs to a triangle in $G[D_r]$. Suppose that some $y\in N_{\vec{G}}^-(x)\cap D_r$ belongs to no triangle in $G[D_r]$. Then $\tau_G(yx)=0$, because every triangle containing $yx$ lies in $G[D_r]$. By Inequality~\eqref{eq:global-edge-count} and the choice of $x$, we would have
\[
0=\tau_G(yx)\ge t+d_{\vec{G}}^+(x)-d_{\vec{G}}^+(y)\ge t,
\]
a contradiction.

Let $R=\bigcup_{s=1}^{r-1}D_s$. The acyclic ordering implies $N_{\vec{G}}^-(x)\setminus D_r\subseteq R\cap V_{j-1}$. Combining this with $|N_{\vec{G}}^-(x)\cap D_r|\le \tau(D_r)$, we have $d_{\vec{G}}^-(x)\le |R\cap V_{j-1}|+\tau(D_r)$, and therefore
\[
d_{\vec{G}}^+(x)\ge n+t-|R\cap V_{j-1}|-\tau(D_r).
\]
By Inequality~\eqref{eq:component-spread}, every $u\in D_r$ satisfies
\[
d_{\vec{G}}^+(u)\ge d_{\vec{G}}^+(x)-(2\tau(D_r)-t)
\ge n+2t-|R\cap V_{j-1}|-3\tau(D_r).
\]

No arc is directed from $D_r$ to $R$. Thus, for every $u\in D_r\cap V_i$, $N_{\vec{G}}^+(u)\subseteq V_{i+1}\setminus R$. This set has $n-|R\cap V_{i+1}|$ vertices and contains $D_r\cap V_{i+1}$. Consequently, the number of nonneighbors of $u$ in $D_r\cap V_{i+1}$ is at most
\[
\begin{aligned}
(n-|R\cap V_{i+1}|)-d_{\vec{G}}^+(u)
&\le |R\cap V_{j-1}|-|R\cap V_{i+1}|+3\tau(D_r)-2t\\
&\le3\tau(D_r)-t-1\le3\tau(D_r)-2.
\end{aligned}
\]
Here $|R\cap V_{j-1}|-|R\cap V_{i+1}|\le t-1$ follows from our assumption on initial segments, and the last inequality uses $t\ge1$. This completes the proof.
\end{proof}

\begin{claim}\label{clm:component-size}
For each $r\in\{1,\ldots,k\}$, $|D_r|\le9\tau(D_r)^2+6\tau(D_r)-6$.
\end{claim}

\begin{proof}
For convenience, let $b=3\tau(D_r)-2$. By Claim~\ref{clm:degree-spread}, we have $b \ge1$. By Claim~\ref{clm:few-nonneighbors}, every vertex of $D_r\cap V_i$ has at most $b$ nonneighbors in $D_r\cap V_{i+1}$.

For $i=0,1,2$, let $X_i$ be the set of vertices in $D_r\cap V_i$ that belong to no triangle in $G[D_r]$. We bound $|X_i|$ by counting the edges between $X_i$ and $D_r\cap V_{i+1}$.

By Claim~\ref{clm:nontrivial-components}, the strong component with vertex set $D_r$ has at least two vertices, so every vertex in $D_r$ has an out-neighbor in $D_r$. For $y\in D_r\cap V_{i+1}$, choose an out-neighbor $z\in D_r\cap V_{i-1}$. If $v\in X_i$ is adjacent to $y$, then $v$ cannot be adjacent to $z$, since otherwise $vyz$ would be a triangle in $G[D_r]$. The vertex $z$ has at most $b$ nonneighbors in $D_r\cap V_i$, so $y$ has at most $b$ neighbors in $X_i$. Summing over $y\in D_r\cap V_{i+1}$, we have
\[
e_G(X_i,D_r\cap V_{i+1})\le b |D_r\cap V_{i+1}|.
\]
On the other hand, by Claim~\ref{clm:nontrivial-components} and strong connectivity, every vertex of $X_i$ has an out-neighbor in $D_r$. By the definition of $\vec{G}$, this out-neighbor lies in $D_r\cap V_{i+1}$, so $e_G(X_i,D_r\cap V_{i+1})\ge |X_i|$. Moreover, by Claim~\ref{clm:few-nonneighbors}, every vertex of $X_i$ has at most $b$ nonneighbors in $D_r\cap V_{i+1}$, and hence at least $|D_r\cap V_{i+1}|-b$ out-neighbors there. Counting these edges from their ends in $X_i$, we obtain the second lower bound $e_G(X_i,D_r\cap V_{i+1})\ge |X_i|(|D_r\cap V_{i+1}|-b)$. Comparing each of these lower bounds with the upper bound above, we have
\begin{equation}\label{eq:nontriangle-counts} 
|X_i|\le b|D_r\cap V_{i+1}|, \quad 
|X_i|\bigl(|D_r\cap V_{i+1}|-b\bigr)\le b|D_r\cap V_{i+1}|. 
\end{equation}
If $|D_r\cap V_{i+1}|\le b$, the first of Inequalities~\eqref{eq:nontriangle-counts} implies $|X_i|\le b^2$. If $|D_r\cap V_{i+1}|>b$, then $|D_r\cap V_{i+1}|-b$ is a positive integer. By the second inequality,
\[
|X_i|\le\frac{b|D_r\cap V_{i+1}|}{|D_r\cap V_{i+1}|-b} =b+\frac{b^2}{|D_r\cap V_{i+1}|-b}\le b+b^2=b(b+1).
\]
Hence $|X_i|\le b(b+1)$ in both cases.

Moreover, if $|X_i|>2b$, the second of Inequalities~\eqref{eq:nontriangle-counts} forces $|D_r\cap V_{i+1}|<2b$, so $|X_{i+1}|\le |D_r\cap V_{i+1}|<2b$. Therefore at most one of $|X_0|,|X_1|,|X_2|$ can exceed $2b$, and
\[
|X_0|+|X_1|+|X_2|\le b(b+1)+4b=b^2+5b.
\]
By definition, every vertex in $D_r\setminus(X_0\cup X_1\cup X_2)$ belongs to a triangle of $G[D_r]$. The $\tau(D_r)$ triangles together contain at most $3\tau(D_r)$ distinct vertices. Hence
\[
|D_r|\le |X_0|+|X_1|+|X_2|+3\tau(D_r)\le3\tau(D_r)+b^2+5b=9\tau(D_r)^2+6\tau(D_r)-6.\qedhere
\]
\end{proof}

Claim~\ref{clm:component-size} bounds the size of every $D_r$. Every triangle lies in exactly one strong component of $\vec{G}$, so $\sum_{r=1}^k\tau(D_r)=\tau(G)$. Since $D_1,\ldots,D_k$ partition $V(G)$, we have
\[
3n=\sum_{r=1}^k|D_r|
\le9\sum_{r=1}^k\tau(D_r)^2+6\tau(G)-6k
\le9\tau(G)^2+6\tau(G)-6,
\]
where we used $\sum_{r=1}^k\tau(D_r)^2\le(\sum_{r=1}^k\tau(D_r))^2$ and $k\ge1$. This contradicts our assumption that $n>3\tau(G)^2+2\tau(G)-2$. Hence the required initial segment exists, and its associated partition is valid by Claim~\ref{clm:initial-segment-partition}. This completes the proof of the lemma.
\end{proof}

\section{Proof of Theorem~\ref{thm:lower}}\label{sec:local}

The partition lemma reduces Theorem~\ref{thm:lower} to the following result for a tripartite graph whose parts need not have equal sizes.

\begin{theorem}\label{thm:local}
Let $t$ be a positive integer, and let $H=(V_0,V_1,V_2)$ be a tripartite graph with nonempty parts ordered as $(V_0,V_1,V_2)$. Suppose that $(V_0,V_1,V_2)$ is $t$-unbalanced and that, for each $i\in\{0,1,2\}$ and every $u\in V_i$, $d_H(u)\ge |V_{i+1}|+t$. Then $\tau(H)\ge6t^3/5$.
\end{theorem}

We first show how this theorem completes the proof of Theorem~\ref{thm:lower}.

\begin{proof}[Proof of Theorem~\ref{thm:lower} (Assuming Theorem~\ref{thm:local})]
Let $t,n$ be integers with $t\ge2$ and $n\ge18t^6$. Let $G=(V_0,V_1,V_2)\in\mathcal{G}_3(n)$ have minimum degree at least $n+t$. Suppose, to the contrary, that $\tau(G)<12t^3/5$. We first check the size condition in Lemma~\ref{lem:prefix}. Since $t^3\ge8$,
\[
3\tau(G)^2+2\tau(G)-2
<\frac{432}{25}t^6+\frac{24}{5}t^3
\le\frac{447}{25}t^6
<18t^6\le n.
\]
Apply Lemma~\ref{lem:prefix} to obtain a valid partition of $G$. For $i=0,1,2$, write $V_i=V_i^1\cup V_i^2$, and let $G_1,G_2$ be the induced graphs defined in the lemma.

By Conditions~\ref{item:partition-nonempty} and~\ref{item:partition-imbalance}, the parts of each graph are nonempty and form a $t$-unbalanced triple. To apply Theorem~\ref{thm:local}, we take the parts of $G_1$ in the order $(V_0^1,V_1^1,V_2^1)$ and those of $G_2$ in the order $(V_0^2,V_2^2,V_1^2)$. With these orders, Conditions~\ref{item:partition-forward} and~\ref{item:partition-backward} are precisely the required degree conditions. Applying Theorem~\ref{thm:local} to both graphs, we have
\[
\tau(G)\ge \tau(G_1)+\tau(G_2)\ge\frac65t^3+\frac65t^3=\frac{12}{5}t^3,
\]
where the first inequality holds because $G_1$ and $G_2$ are vertex-disjoint. This contradicts the assumed upper bound on $\tau(G)$ and completes the proof.
\end{proof}

In the rest of this section, let $t$ be a fixed positive integer and let $H=(V_0,V_1,V_2)$ satisfy the hypotheses of Theorem~\ref{thm:local}. We take the parts of $H$ in the order $(V_0,V_1,V_2)$ and use the corresponding orientation $\vec{H}$ defined in Section~\ref{sec:decomposition}.

For each $i\in\{0,1,2\}$ and every $u\in V_i$,
\[
d_{\vec{H}}^-(u)=d_H(u)-d_{\vec{H}}^+(u)\ge |V_{i+1}|+t-d_{\vec{H}}^+(u)\ge t.
\]
Since $N_{\vec{H}}^-(u)\subseteq V_{i-1}$, each part has at least $t$ vertices. For each $i\in\{0,1,2\}$, let $S_i\subseteq V_i$ be a set of $t$ vertices with the largest out-degrees in $\vec{H}$. If some vertices have the same out-degree, we choose any of them as needed. Let $S=S_0\cup S_1\cup S_2$.

For each $i \in \{0,1,2,3\}$, let $\tau_i$ denote the number of triangles of $H$ with exactly $i$ vertices in $S$. Thus $\tau_0=\tau(H-S)$, $\tau_3=\tau(S)$, and $\tau(H)=\tau_0+\tau_1+\tau_2+\tau_3$.

Our proof is by induction on $t$. Deleting $S$ alone may leave some vertices with too few neighbors for induction to apply. We therefore delete some additional vertices and apply the induction hypothesis to the remaining graph to obtain a contradiction. In Subsection~\ref{subsec:triangle-counts}, we count triangles meeting $S$. In Subsection~\ref{subsec:deletions}, we choose additional vertices to delete. In Subsection~\ref{subsec:induction}, we complete the induction and hence the proof of Theorem~\ref{thm:lower}.

\subsection{Counting triangles meeting \texorpdfstring{$S$}{S}}\label{subsec:triangle-counts}

For $U\subseteq V(H)$ and $v\in V(H)\setminus U$, let 
$$\bar{d}_U(v)=|U|-d_U(v)$$ denote the number of nonneighbors of $v$ in $U$. For each $i\in\{0,1,2\}$ and every $v\in V_i$, the degree condition implies $d_{\vec{H}}^-(v)\ge t+\bar{d}_{V_{i+1}}(v)$.
Thus for each $u\in S_i$, we have
\begin{equation}\label{eq:outside-indegree}
d_{V_{i-1}\setminus S_{i-1}}(u)=d_{\vec{H}}^-(u)-d_{S_{i-1}}(u)\ge t+\bar{d}_{V_{i+1}}(u) -d_{S_{i-1}}(u)  = \bar{d}_{S_{i-1}}(u)+\bar{d}_{V_{i+1}}(u).
\end{equation}

For an arc $(u,v)$ of $\vec{H}$ with $u\in V_i$, we have $N_{\vec{H}}^-(u)\cup N_{\vec{H}}^+(v)\subseteq V_{i-1}$ and $\tau_H(uv)=|N_{\vec{H}}^-(u)\cap N_{\vec{H}}^+(v)|$. Hence
\begin{equation}\label{eq:local-edge}
\tau_H(uv)\ge d_{\vec{H}}^-(u)+d_{\vec{H}}^+(v)-|V_{i-1}|\ge t+\bar{d}_{V_{i+1}}(u)-\bar{d}_{V_{i-1}}(v).
\end{equation}

We first count triangles with exactly one vertex in $S$. 
\begin{lemma}\label{lem:two-outside}
The number $\tau_1$ of triangles with exactly one vertex in $S$ satisfies
\[
\tau_1\ge\sum_i\sum_{u\in S_i}\bigl[\bar{d}_{S_{i-1}}(u)\bar{d}_{S_{i+1}}(u)+\bigl(\bar{d}_{S_{i+1}}(u)-\bar{d}_{S_{i-1}}(u)\bigr)\bar{d}_{V_{i+1}}(u)\bigr].
\]
\end{lemma}

\begin{proof}
Let $i\in\{0,1,2\}$ and $u\in S_i$. For each nonneighbor $x\in S_{i-1}$ of $u$, we choose a neighbor $m_x\in V_{i-1}\setminus S_{i-1}$ of $u$ as the mate of $x$, so that the vertices $m_x$ are pairwise distinct. Such a choice is possible because $d_{V_{i-1}\setminus S_{i-1}}(u)\ge\bar{d}_{S_{i-1}}(u)$ by Inequality~\eqref{eq:outside-indegree}.

For each such $x$, the choice of $S_{i-1}$ ensures $\bar{d}_{V_i}(m_x)\ge\bar{d}_{V_i}(x)$. Applying Inequality~\eqref{eq:local-edge} to the arc $(m_x,u)$, we have
\[
\tau_H(um_x)\ge t+\bar{d}_{V_i}(m_x)-\bar{d}_{V_{i+1}}(u)\ge t+\bar{d}_{V_i}(x)-\bar{d}_{V_{i+1}}(u).
\]
At most $d_{S_{i+1}}(u)$ of these triangles have their third vertex in $S_{i+1}$. Thus for every $u\in S_i$ and $x\in S_{i-1}\setminus N_H(u)$, the number of triangles containing $um_x$ and having their third vertex outside $S$ is at least
\begin{equation}\label{eq:paired-edge-triangles}
t+\bar{d}_{V_i}(x)-\bar{d}_{V_{i+1}}(u) - d_{S_{i+1}}(u) =  \bar{d}_{V_i}(x)+\bar{d}_{S_{i+1}}(u)-\bar{d}_{V_{i+1}}(u).
\end{equation}
Each such triangle has exactly one vertex in $S$.

Each triangle counted in~\eqref{eq:paired-edge-triangles} uniquely determines its vertex $u$ in $S$ and its vertex $m_x$ in the preceding part. Since the vertices $m_x$ chosen for a fixed $u$ are pairwise distinct, no triangle is counted twice. Summing the lower bound in~\eqref{eq:paired-edge-triangles} over all $i$, $u\in S_i$ and $x\in S_{i-1}\setminus N_H(u)$, we obtain
\[
\tau_1\ge\sum_i\sum_{\substack{u\in S_i,\ x\in S_{i-1}\\xu\notin E(H)}}\bigl(\bar{d}_{V_i}(x)+\bar{d}_{S_{i+1}}(u)-\bar{d}_{V_{i+1}}(u)\bigr).
\]
For each $x\in S_{i-1}$, the term $\bar{d}_{V_i}(x)$ is included once for every nonneighbor $u\in S_i$ of $x$, and hence appears $\bar{d}_{S_i}(x)$ times. Similarly, for each $u\in S_i$, the expression $\bar{d}_{S_{i+1}}(u)-\bar{d}_{V_{i+1}}(u)$ is included once for every nonneighbor $x\in S_{i-1}$ of $u$, and hence appears $\bar{d}_{S_{i-1}}(u)$ times. The double sum therefore equals
\[
\sum_i\sum_{x\in S_{i-1}}\bar{d}_{S_i}(x)\bar{d}_{V_i}(x)+\sum_i\sum_{u\in S_i}\bar{d}_{S_{i-1}}(u)\bigl(\bar{d}_{S_{i+1}}(u)-\bar{d}_{V_{i+1}}(u)\bigr).
\]
We relabel $i-1$ as $i$ and $x$ as $u$ in the first sum and combine the two sums to obtain the stated bound.
\end{proof}

We next count triangles with exactly two vertices in $S$.
For disjoint sets $A,B\subseteq V(H)$, let $$\bar{e}_H(A,B)=|A||B|-e_H(A,B)$$ denote the number of nonedges between $A$ and $B$. For each $i\in\{0,1,2\}$, counting these nonedges from their endpoints in $S_i$, we have
\[
\sum_{u\in S_i}\bar{d}_{S_{i-1}}(u)=\bar{e}_H(S_{i-1},S_i),\quad \sum_{u\in S_i}\bar{d}_{S_{i+1}}(u)=\bar{e}_H(S_i,S_{i+1}).
\]

\begin{lemma}\label{lem:one-outside}
The number $\tau_2$ of triangles with exactly two vertices in $S$ satisfies
\[
\tau_2\ge2t\sum_i\bar{e}_H(S_i,S_{i+1})+\sum_i\sum_{u\in S_i}\bigl[\bigl(\bar{d}_{S_{i-1}}(u)-\bar{d}_{S_{i+1}}(u)\bigr)\bar{d}_{V_{i+1}}(u)-2\bar{d}_{S_{i-1}}(u)\bar{d}_{S_{i+1}}(u)\bigr].
\]
\end{lemma}

\begin{proof}
We count pairs $((u,v),T)$, where $(u,v)$ is an arc with both endpoints in $S$ and $T$ is a triangle with exactly two vertices in $S$ containing the edge $uv$. Each such triangle contains exactly one arc with both endpoints in $S$, so it is counted exactly once and there are $\tau_2$ such pairs.

Let $(u,v)$ be an arc from $S_i$ to $S_{i+1}$. We have $N_{\vec{H}}^-(u)\cup N_{\vec{H}}^+(v)\subseteq V_{i-1}$. By Inequality~\eqref{eq:outside-indegree}, the vertex $u$ has at least $\bar{d}_{S_{i-1}}(u)+\bar{d}_{V_{i+1}}(u)$ in-neighbors in $V_{i-1}\setminus S_{i-1}$. The vertex $v$ has $\bar{d}_{V_{i-1}}(v)-\bar{d}_{S_{i-1}}(v)$ nonneighbors there. Therefore there are at least
\[
\bar{d}_{S_{i-1}}(u)+\bar{d}_{V_{i+1}}(u)-\bigl(\bar{d}_{V_{i-1}}(v)-\bar{d}_{S_{i-1}}(v)\bigr)=\bar{d}_{S_{i-1}}(u)+\bar{d}_{V_{i+1}}(u)+\bar{d}_{S_{i-1}}(v)-\bar{d}_{V_{i-1}}(v)
\]
triangles containing $uv$ whose third vertex lies in $V_{i-1}\setminus S_{i-1}$.

\begin{samepage}
For each $u\in S_i$, there are $d_{S_{i+1}}(u)$ such arcs with tail $u$ and $d_{S_{i-1}}(u)$ such arcs with head $u$. Its contributions at these two ends are $\bar{d}_{S_{i-1}}(u)+\bar{d}_{V_{i+1}}(u)$ and $\bar{d}_{S_{i+1}}(u)-\bar{d}_{V_{i+1}}(u)$, respectively. Counting by arcs, we obtain the following lower bound on $\tau_2$:
\begin{align*}
&\sum_i\sum_{u\in S_i}\bigl[\bigl(\bar{d}_{S_{i-1}}(u)+\bar{d}_{V_{i+1}}(u)\bigr)d_{S_{i+1}}(u)+\bigl(\bar{d}_{S_{i+1}}(u)-\bar{d}_{V_{i+1}}(u)\bigr)d_{S_{i-1}}(u)\bigr]\\*
=&\sum_i\sum_{u\in S_i}\bigl[\bigl(\bar{d}_{S_{i-1}}(u)+\bar{d}_{V_{i+1}}(u)\bigr)\bigl(t-\bar{d}_{S_{i+1}}(u)\bigr)+\bigl(\bar{d}_{S_{i+1}}(u)-\bar{d}_{V_{i+1}}(u)\bigr)\bigl(t-\bar{d}_{S_{i-1}}(u)\bigr)\bigr]\\*
=&\sum_i\sum_{u\in S_i}\bigl[t\bigl(\bar{d}_{S_{i-1}}(u)+\bar{d}_{S_{i+1}}(u)\bigr)+\bigl(\bar{d}_{S_{i-1}}(u)-\bar{d}_{S_{i+1}}(u)\bigr)\bar{d}_{V_{i+1}}(u)-2\bar{d}_{S_{i-1}}(u)\bar{d}_{S_{i+1}}(u)\bigr]\\*
=&2t\sum_i\bar{e}_H(S_i,S_{i+1})+\sum_i\sum_{u\in S_i}\bigl[\bigl(\bar{d}_{S_{i-1}}(u)-\bar{d}_{S_{i+1}}(u)\bigr)\bar{d}_{V_{i+1}}(u)-2\bar{d}_{S_{i-1}}(u)\bar{d}_{S_{i+1}}(u)\bigr].\qedhere
\end{align*}
\end{samepage}
\end{proof}

We now combine the bounds on $\tau_1$ and $\tau_2$ in Lemmas~\ref{lem:two-outside} and~\ref{lem:one-outside} to estimate $\tau_1+\tau_2+\tau_3$, the number of triangles meeting $S$.

\begin{lemma}\label{lem:core}
$\tau_1+\tau_2+\tau_3\ge t^3+t\max_j\bar{e}_H(S_j,S\setminus S_j)$ and $\tau_2\le2\big(\tau_1+\tau_2+\tau_3-t^3\big)$.
\end{lemma}

\begin{proof}
For each $r\in\{1,2,3\}$, let $c_r$ denote the number of triples in $S_0\times S_1\times S_2$ with exactly $r$ missing edges. Then
\begin{equation}\label{eq:nontriangle-triples}
c_1+c_2+c_3=t^3-\tau_3.
\end{equation}

We firstly double count pairs $(\{x,y\},z)$, where $xy\notin E(H)$ and $x,y,z$ lie in distinct sets among $S_0,S_1,S_2$. For each such $\{x,y\}$, there are $t$ choices for $z$. On the other hand, each triple counted by $c_r$ determines exactly $r$ such pairs. Therefore
\[
t\sum_i\bar{e}_H(S_i,S_{i+1})=c_1+2c_2+3c_3.
\]

We next double-count pairs $(\{x,y\},z)$, where $xz,yz\notin E(H)$ and $x,y,z$ lie in distinct sets among $S_0,S_1,S_2$. For each $z\in S_i$, there are $\bar{d}_{S_{i-1}}(z)\bar{d}_{S_{i+1}}(z)$ choices for $\{x,y\}$. On the other hand, each triple counted by $c_1,c_2$ or $c_3$ determines zero, one or three such pairs, respectively. Therefore
\[
\sum_i\sum_{z\in S_i}\bar{d}_{S_{i-1}}(z)\bar{d}_{S_{i+1}}(z)=c_2+3c_3.
\]

When we add the bounds in Lemmas~\ref{lem:two-outside} and~\ref{lem:one-outside}, the terms involving $\bar{d}_{V_{i+1}}(u)$ cancel. Including the $\tau_3$ triangles contained in $S$ and using Equality~\eqref{eq:nontriangle-triples}, we obtain
\begin{equation}\label{eq:mixed}
\begin{aligned}
\tau_1+\tau_2+\tau_3\ge & \tau_3+2t\sum_i\bar{e}_H(S_i,S_{i+1})-\sum_i\sum_{z\in S_i}\bar{d}_{S_{i-1}}(z)\bar{d}_{S_{i+1}}(z)\\
= & \tau_3+2(c_1+2c_2+3c_3)-(c_2+3c_3)\\
= & 2t^3-\tau_3+c_2+c_3 \\
\ge & 2t^3-\tau_3.
\end{aligned}
\end{equation}

Let $i\in\{0,1,2\}$. We now double count pairs $(\{x,y\},z)$, where $xy\notin E(H)$, $\{x,y\}\cap S_i\ne\varnothing$, and $x,y,z$ lie in distinct sets among $S_0,S_1,S_2$. For each such $\{x,y\}$, there are $t$ choices for $z$. On the other hand, a triple counted by $c_1$ determines at most one such pair, and every triple determines at most two, since only two choices of $\{x,y\}$ meet $S_i$. Taking the maximum over $i$, we obtain
\begin{equation}\label{eq:restricted-nonedges}
t\max_j\bar{e}_H(S_j,S\setminus S_j)\le c_1+2c_2+2c_3.
\end{equation}
By the second line of Inequality~\eqref{eq:mixed}, together with Inequality~\eqref{eq:restricted-nonedges} and Equality~\eqref{eq:nontriangle-triples}, we also have
\[
\tau_1+\tau_2+\tau_3\ge\tau_3+2c_1+3c_2+3c_3=t^3+c_1+2c_2+2c_3\ge t^3+t\max_j\bar{e}_H(S_j,S\setminus S_j).
\]
This proves the first bound.

Finally, by Inequality~\eqref{eq:mixed}, $\tau_3\ge2t^3-\big(\tau_1+\tau_2+\tau_3\big)$, so
\[
\tau_2\le\tau_1+\tau_2\le2\big(\tau_1+\tau_2+\tau_3-t^3\big).
\]
This proves the stated upper bound.
\end{proof}

\subsection{Deleting additional vertices}\label{subsec:deletions}
Deleting $S$ may remove too many neighbors from some remaining vertices for induction to apply. We use the triangles counted by $\tau_2$ to choose additional vertices to delete, then verify that the remaining graph satisfies the conditions needed for induction.

Let $\eta$ be a real number with $0<\eta<1$. For each $i\in\{0,1,2\}$, let 
\begin{equation}\label{eq:opposite-density}
\hat{\rho}_i=\frac{e_H(S_{i-1},S_{i+1})}{t^2}
\end{equation}
 denote the edge density between the two sets other than $S_i$. For each $i\in\{0,1,2\}$ and every real number $r\in[0,1-\eta]$, define
\[
D_i(r)=\{u\in V_i\setminus S_i\mid d_S(u)>(2-\eta-r)t\}.
\]
The function $|D_i(r)|$ is a nondecreasing step function of $r$, as illustrated in Fig.~\ref{fig:deletion-integral}.

\begin{figure}[htbp]
\centering
\begin{tikzpicture}[x=1cm,y=0.34cm,font=\small]
  \fill[gray!13] (0.9,0) rectangle (7.6,1);
  \fill[gray!13] (2.6,1) rectangle (7.6,4);
  \fill[gray!13] (4.3,4) rectangle (7.6,6);
  \fill[gray!13] (5.8,6) rectangle (7.6,7);
  \fill[cyan!30] (2.6,2) rectangle (7.6,3);
  \draw[gray!55,line width=0.4pt] (7.6,0)--(7.6,7);
  \draw[gray!40,line width=0.4pt] (2.6,1)--(7.6,1)
    (4.3,4)--(7.6,4) (5.8,6)--(7.6,6);
  \draw[->,line width=0.6pt] (0,0)--(8.35,0) node[right] {$r$};
  \draw[->,line width=0.6pt] (0,0)--(0,7.9) node[above] {$|D_i(r)|$};
  \node[below left] at (0,0) {$0$};
  \draw[gray!65,dashed,line width=0.45pt] (2.6,0)--(2.6,4);
  \draw[line width=0.8pt] (0,0)--(0.9,0) (0.9,1)--(2.6,1)
    (2.6,4)--(4.3,4) (4.3,6)--(5.8,6) (5.8,7)--(7.6,7);
  \foreach \x/\y in {0.9/0,2.6/1,4.3/4,5.8/6,7.6/7}
    \fill (\x,\y) circle[radius=1.2pt];
  \foreach \x/\y in {0.9/1,2.6/4,4.3/6,5.8/7}
    \draw[fill=white,line width=0.6pt] (\x,\y) circle[radius=1.2pt];
  \node at (5.1,2.5) {$u$};
  \draw[<->,line width=0.5pt] (7.92,2)--(7.92,3) node[midway,right] {$1$};
  \draw[line width=0.5pt] (2.6,0)--(2.6,-0.18) (7.6,0)--(7.6,-0.18);
  \node[below=3pt] at (2.6,0) {$2-\eta-d_S(u)/t$};
  \node[below=3pt] at (7.6,0) {$\hat{\rho}_i-\eta$};
\end{tikzpicture}
\caption{Illustration of the function $|D_i(r)|$.}
\label{fig:deletion-integral}
\end{figure}
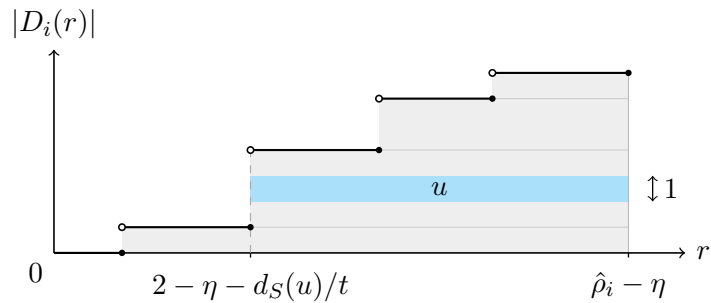

The following lemma provides a suitable value for each part.

\begin{lemma}\label{lem:deletion-thresholds}
Let $\eta$ be a real number with $0<\eta<1$. There exist real numbers $r_0,r_1,r_2$ such that, for each $i\in\{0,1,2\}$, $r_i\in[0,\hat{\rho}_i-\eta]$ and $|D_i(r_i)|<r_{i+1}t$, provided that
\begin{equation}\label{eq:deletion-condition}
\frac{\tau_2}{t^3}<(1-\eta)\left(1-\eta-\frac{\max_j\bar{e}_H(S_j,S\setminus S_j)}{t^2}\right).
\end{equation}
\end{lemma}

\begin{proof}
We first bound the integrals of $|D_i(r)|$, then use a volume estimate to choose $r_0,r_1,r_2$. For each $i\in\{0,1,2\}$, Inequality~\eqref{eq:deletion-condition} implies $\eta<1-\max_j\bar{e}_H(S_j,S\setminus S_j)/t^2\le\hat{\rho}_i$, so the interval $[0,\hat{\rho}_i-\eta]$ has positive length. We first prove that
\begin{equation}\label{eq:deletion-integral}
\sum_i\int_0^{\hat{\rho}_i-\eta}|D_i(r)|\,dr\le\frac{\tau_2}{t^2}.
\end{equation}
Let $i\in\{0,1,2\}$ and $u\in V_i\setminus S_i$. By definition, $u\in D_i(r)$ exactly when $r>2-\eta-d_S(u)/t$. If $d_S(u)\le(2-\hat{\rho}_i)t$, then $u$ belongs to none of the sets $D_i(r)$ in the stated range. Otherwise, the values of $r$ for which $u\in D_i(r)$ form an interval of length at most $d_S(u)/t-2+\hat{\rho}_i$. Among the edges between $S_{i-1}$ and $S_{i+1}$, at most $t(2t-d_S(u))$ have an endpoint not adjacent to $u$, since each of the $2t-d_S(u)$ nonneighbors is incident with at most $t$ such edges. Every other edge forms a triangle with $u$. By Equality~\eqref{eq:opposite-density}, the number of these triangles is at least
\[
e_H(S_{i-1},S_{i+1})-t(2t-d_S(u))=t^2\left(\frac{d_S(u)}{t}-2+\hat{\rho}_i\right).
\]
Thus the number of triangles with unique vertex $u$ outside $S$ is at least $t^2$ times the length of the interval for which $u\in D_i(r)$. The integral of $|D_i(r)|$ is the sum of these lengths over $u\in V_i\setminus S_i$ (see Fig.~\ref{fig:deletion-integral}). Since every triangle counted by $\tau_2$ has a unique vertex outside $S$, summing over the three parts proves Inequality~\eqref{eq:deletion-integral}.

We write $\operatorname{vol}(U)$ for the volume of a region $U\subseteq\mathbb{R}^3$. Let $B$ be the box $\prod_{i=0}^2[0,\hat{\rho}_i-\eta]$. For each $i\in\{0,1,2\}$, let $B_i$ be the set of triples $(r_0,r_1,r_2)\in B$ satisfying $|D_i(r_i)|\ge r_{i+1}t$. We will show that $\operatorname{vol}(\bigcup_{i=0}^2 B_i)<\operatorname{vol}(B)$.

For each $i\in\{0,1,2\}$,
\[
\begin{aligned}
\operatorname{vol}(B_i)
=&\int_0^{\hat{\rho}_i-\eta}\int_0^{\min\{\hat{\rho}_{i+1}-\eta,\,|D_i(r_i)|/t\}}\int_0^{\hat{\rho}_{i-1}-\eta}1\,dr_{i-1}\,dr_{i+1}\,dr_i\\
=&(\hat{\rho}_{i-1}-\eta)\int_0^{\hat{\rho}_i-\eta}\min\left\{\hat{\rho}_{i+1}-\eta,\frac{|D_i(r)|}{t}\right\}\,dr\\
\le&\frac{\hat{\rho}_{i-1}-\eta}{t}\int_0^{\hat{\rho}_i-\eta}|D_i(r)|\,dr.
\end{aligned}
\]
By cyclically relabelling the indices, we may assume that $\hat{\rho}_0=\max_i\hat{\rho}_i$. By Inequality~\eqref{eq:deletion-integral},
\begin{equation}\label{eq:bad-region-volume} 
\operatorname{vol}\left(\bigcup_{i=0}^2 B_i\right)
\le \sum_i\operatorname{vol}(B_i) 
\le \frac{\hat{\rho}_0-\eta}{t}\sum_i\int_0^{\hat{\rho}_i-\eta}|D_i(r)|\,dr 
\le (\hat{\rho}_0-\eta)\frac{\tau_2}{t^3}. 
\end{equation}
Since $\hat{\rho}_1,\hat{\rho}_2\le1$ and $1-\eta>0$, we have
\[
\begin{aligned}
(\hat{\rho}_1-\eta)(\hat{\rho}_2-\eta)
=&(1-\eta)(\hat{\rho}_1+\hat{\rho}_2-1-\eta)+(1-\hat{\rho}_1)(1-\hat{\rho}_2)\\
\ge&(1-\eta)(\hat{\rho}_1+\hat{\rho}_2-1-\eta)\\
=&(1-\eta)\left(1-\eta-\frac{\bar{e}_H(S_0,S\setminus S_0)}{t^2}\right)\\
\ge&(1-\eta)\left(1-\eta-\frac{\max_j\bar{e}_H(S_j,S\setminus S_j)}{t^2}\right) >\frac{\tau_2}{t^3},
\end{aligned}
\]
where the second equality follows from Equality~\eqref{eq:opposite-density} and the last inequality follows from Inequality~\eqref{eq:deletion-condition}.
By Inequality~\eqref{eq:bad-region-volume}, $\operatorname{vol}(\bigcup_{i=0}^2 B_i)<\prod_{i=0}^2(\hat{\rho}_i-\eta)=\operatorname{vol}(B)$. Hence we can choose $(r_0,r_1,r_2)\in B\setminus\bigcup_{i=0}^2 B_i$. For each $i\in\{0,1,2\}$, we have $|D_i(r_i)|<r_{i+1}t$, as required.
\end{proof}

We now verify that the remaining graph satisfies the conditions needed for induction.

\begin{lemma}\label{lem:remaining-graph}
Let $\eta$ be a real number with $0<\eta<1$, and let $r_0,r_1,r_2\in[0,1-\eta]$ be fixed. For each $i\in\{0,1,2\}$, let $R_i=V_i\setminus\bigl(S_i\cup D_i(r_i)\bigr)$, and suppose that $|D_i(r_i)|<r_{i+1}t$. Let $R=H[R_0\cup R_1\cup R_2]$. Then $R_0,R_1,R_2$ are nonempty and form an $\eta t$-unbalanced triple. Moreover, for each $i\in\{0,1,2\}$ and every $u\in R_i$, $d_R(u)>|R_{i+1}|+\eta t$.
\end{lemma}

\begin{proof}
For each $i\in\{0,1,2\}$, we have $|R_i|=|V_i|-t-|D_i(r_i)|$ and $|D_i(r_i)|<r_{i+1}t\le(1-\eta)t$. By cyclically relabelling the indices, we may assume that $|V_0|=\max_i|V_i|$. Since $(V_0,V_1,V_2)$ is $t$-unbalanced,
\[
|R_0|-\min_i|R_i|\ge |V_0|-\min_i|V_i|-|D_0(r_0)|\ge t-|D_0(r_0)|>\eta t.
\]
Thus $R$ is nonempty and $(R_0,R_1,R_2)$ is $\eta t$-unbalanced.

Let $i\in\{0,1,2\}$ and $u\in R_i$. Since $u\notin D_i(r_i)$, we have $d_S(u)\le(2-\eta-r_i)t$. Besides these neighbors in $S$, the vertex $u$ loses at most $|D_{i-1}(r_{i-1})|+|D_{i+1}(r_{i+1})|$ neighbors in the additional deletions. Since $d_H(u)\ge |V_{i+1}|+t$,
\[
\begin{aligned}
d_R(u)
&\ge |V_{i+1}|+t-(2-\eta-r_i)t-|D_{i-1}(r_{i-1})|-|D_{i+1}(r_{i+1})|\\
&=|R_{i+1}|+(\eta+r_i)t-|D_{i-1}(r_{i-1})|\\
&>|R_{i+1}|+\eta t,
\end{aligned}
\]
where the last inequality uses $|D_{i-1}(r_{i-1})|<r_i t$.

This degree bound implies that every vertex in $R_i$ has a neighbor in $R_{i-1}$. Since $R$ is nonempty, applying this observation twice shows that all three parts of $R$ are nonempty.
\end{proof}

\subsection{Proof of Theorem~\ref{thm:local}}\label{subsec:induction}

In this subsection, we prove Theorem~\ref{thm:local} by induction on $t$.

First suppose that $t=1$. Assume for a contradiction that $\tau(H)\le 1$.
Since the three parts are nonempty and $1$-unbalanced, $H$ has at least four vertices. Thus some vertex belongs to no triangle. Among all such vertices, let $x\in V_i$ be one with the fewest nonneighbors in its next part. Every in-neighbor $y\in V_{i-1}$ of $x$ belongs to a triangle. Otherwise, the choice of $x$ and Inequality~\eqref{eq:local-edge} imply
\[
0 = \tau_H(yx) \ge 1 + \bar{d}_{V_i}(y) - \bar{d}_{V_{i+1}}(x) \ge 1,
\]
a contradiction. Since $\tau(H) \le 1$, at most one vertex of $V_{i-1}$ belongs to a triangle, so $d_{\vec{H}}^-(x) \le 1$. The degree condition implies $d_{\vec{H}}^-(x) \ge 1+\bar{d}_{V_{i+1}}(x)$. Thus $x$ has a unique in-neighbor $y$ and $\bar{d}_{V_{i+1}}(x)=0$. By Inequality~\eqref{eq:local-edge},
\[
0 = \tau_H(yx) \ge 1 + \bar{d}_{V_i}(y) \ge 1,
\]
again a contradiction. Therefore $\tau(H) \ge 2>6/5$, which proves the case $t=1$.

Now let $t \ge 2$, and suppose that the theorem holds for every smaller positive integer. Suppose, to the contrary, that $\tau(H)<6t^3/5$.

Recall that $S=S_0\cup S_1\cup S_2$ has been fixed. For each $i \in \{0,1,2,3\}$, $\tau_i$ denotes the number of triangles of $H$ with exactly $i$ vertices in $S$. Let
\[
\beta = \frac{\tau_0}{t^3},\quad \eta = \frac14 + 2\beta.
\]
By Lemma~\ref{lem:core} and the assumption that $\tau(H) < 6t^3/5$, we have
\begin{equation}\label{eq:small}
0 \le \beta + \frac{\max_j\bar{e}_H(S_j,S\setminus S_j)}{t^2} < \frac15,\quad \frac{\tau_2}{t^3} < \frac25-2\beta.
\end{equation}
Thus $1/4 \le \eta < 13/20 < 2/3$. By Inequalities~\eqref{eq:small},
\[
(1-\eta)\left(1-\eta-\frac{\max_j\bar{e}_H(S_j,S\setminus S_j)}{t^2}\right)
> \left(\frac34-2\beta\right)\left(\frac{11}{20}-\beta\right)
= \frac25 - 2\beta + 2\beta^2 + \frac{3\beta}{20} + \frac1{80} > \frac{\tau_2}{t^3}.
\]
By Lemma~\ref{lem:deletion-thresholds}, we can choose $r_0,r_1,r_2$ such that $0 \le r_i \le \hat{\rho}_i-\eta \le 1-\eta$ and $|D_i(r_i)| < r_{i+1}t$, for each $i \in \{0,1,2\}$. Since $(V_0,V_1,V_2)$ is $t$-unbalanced, we may apply Lemma~\ref{lem:remaining-graph} to obtain the induced graph $R\subseteq H-S$ with nonempty parts $R_0, R_1$ and $R_2$.

Let $t'=\lceil\eta t\rceil$. The degrees and part sizes are integers, so Lemma~\ref{lem:remaining-graph} implies that, for each $i \in \{0,1,2\}$ and every $u\in R_i$, $d_R(u) \ge |R_{i+1}|+t'$, and that $(R_0,R_1,R_2)$ is $t'$-unbalanced. Since $\eta > 0$, we have $t' \ge 1$. For $t \ge 3$, we have $\eta t < 2t/3 \le t-1$, so $t' \le t-1$. For $t=2$, the assumed upper bound on $\tau(H)$ implies $\tau(H) \le 9$, and Lemma~\ref{lem:core} implies $\tau_1+\tau_2+\tau_3 \ge 8$. Thus $\tau_0 \le 1$, $\beta \le 1/8$, and $\eta \le 1/2$, so $t'=1$. Therefore $1 \le t' < t$, and the induction hypothesis applies to $R$ with parameter $t'$.

The arithmetic--geometric mean inequality implies
\[
\eta^3=\left(\frac18+\frac18+2\beta\right)^3 \ge \frac{27\beta}{32}.
\]
Thus $\eta^3 > 5\beta/6$ when $\beta > 0$, and the same inequality holds when $\beta=0$ because $\eta=1/4$. By the induction hypothesis,
\[
\tau(R) \ge \frac65(t')^3 \ge \frac65\eta^3t^3 > \beta t^3=\tau_0.
\]
This contradicts $\tau(R) \le \tau_0$, since $R\subseteq H-S$. Thus our assumption that $\tau(H) < 6t^3/5$ is false, and we have $\tau(H) \ge 6t^3/5$. This completes the proof.

\section*{Acknowledgements}

We thank Professor Jie Ma for bringing this problem to our attention and for his helpful comments.
Chunqiu Fang was supported by the National Natural Science Foundation for Young Scientists of China (Grant No.~12301435). Rongxing Xu was supported by the National Natural Science Foundation for Young Scientists of China (Grant No.~12401472) and the Zhejiang Provincial Natural Science Foundation of China (Grant No.~LQN25A010011).

\section*{Declaration of AI use}

The authors used AI tools to assist with developing ideas, checking the proofs, and improving the language. All AI-assisted calculations and suggestions were independently verified and revised by the authors. The authors take full responsibility for the content of the paper.

\end{document}